\documentclass[11pt, a4paper]{article}
\usepackage{amssymb,amsmath,amsthm,tocloft}
\usepackage{mathtools}
\usepackage{fancyhdr}
\usepackage{comment,enumerate}
\usepackage[british]{babel}
\usepackage{enumitem,verbatim}
\usepackage{stackengine}
\usepackage[usenames,dvipsnames,table]{xcolor}
\usepackage{graphicx,tikz,caption,subcaption}
\usetikzlibrary{decorations.pathreplacing}
\usetikzlibrary{decorations.pathmorphing}
\usetikzlibrary{calc,arrows,decorations.markings}
\usetikzlibrary{shapes}
\tikzset{snake it/.style={decorate, decoration=snake}}
\usepackage{bbm,dsfont}

\usepackage{hyperref}
\usepackage{todonotes} % % % % % % % % % % % % % % %
\usepackage[margin=2.75cm]{geometry}

\newtheorem{theorem}{Theorem}[section]
\newtheorem{prop}[theorem]{Proposition}
\newtheorem{lemma}[theorem]{Lemma}
\newtheorem{cor}[theorem]{Corollary}

\newcommand{\repeatlabel}{}
\newtheorem*{repeatlemma}{Lemma \repeatlabel}
\theoremstyle{definition}
\newtheorem{problem}{Problem}

\newtheorem{defn}[theorem]{Definition}

\newtheorem{claim}[theorem]{Claim}

\newenvironment{poc}{\begin{proof}[Proof of claim]}{\end{proof}}

\newcommand{\bP}{\mathrm{Pr}}

\newcommand{\cW}{\mathcal{W}}

\newcommand{\de}{\delta}

\newcommand{\e}{\epsilon}
\newcommand{\n}{\nu}
\newcommand{\h}{\eta}
\newcommand{\G}{\Gamma}
\newcommand{\E}{\mathcal{E}}

\DeclareMathOperator{\mix}{mix}

\title{Well-mixing vertices and almost expanders}

\author{Debsoumya Chakraborti\thanks{
Discrete Mathematics Group, Institute for Basic Science (IBS), Daejeon, South Korea.
E-mail: {\tt \{debsoumya, jinhakim\}@ibs.re.kr}. Supported by the Institute for Basic Science (IBS-R029-C1).
}
\and 
Jaehoon Kim\thanks{Department of Mathematical Sciences, KAIST, South Korea. E-mail: {\tt jaehoon.kim@kaist.ac.kr}. Supported by the POSCO Science Fellowship of POSCO TJ Park Foundation and by the KAIX Challenge program of KAIST Advanced Institute for Science-X.}
\and
Jinha Kim\footnotemark[1]\and
Minki Kim\thanks{
Division of Liberal Arts and Sciences, Gwangju Institute of Science and Technology, Gwangju, South Korea.
E-mail: {\tt minkikim@gist.ac.kr}. Supported by the Institute for Basic Science (IBS-R029-C1).
}
\and
Hong Liu\thanks{Extremal Combinatorics and Probability Group (ECOPRO), Institute for Basic Science (IBS), Daejeon, South Korea, Email: {\texttt hongliu@ibs.re.kr}. Supported by the Institute for Basic Science (IBS-R029-C4) and the UK Research and Innovation Future Leaders Fellowship MR/S016325/1.} }

\begin{document}

\maketitle

\begin{abstract}
We study regular graphs in which the random walks starting from a positive fraction of vertices have small mixing time. We prove that any such graph is virtually an expander and has no small separator. This answers a question of Pak~[\emph{SODA}, 2002]. As a corollary, it shows that sparse (constant degree) regular graphs with many well-mixing vertices have a long cycle, improving a result of Pak. Furthermore, such cycle can be found in polynomial time.

Secondly, we show that if the random walks from a positive fraction of  vertices are well-mixing, then the random walks from almost all vertices are well-mixing (with a slightly worse mixing time).
\end{abstract}

%\tableofcontents
%\date{\today}

%\newpage
%%%%%%%%%%%%%%%%%%%%%%%%%%%%%%%%%%%%%%%
%%%%%%%%%%%%%%%%%%%%%%%%%%%%%%%%%%%%%%%
%%%%%%%%%%%%%%%%%%%%%%%%%%%%%%%%%%%%%%%
%%%%%%%%%%%%%%%%%%%%%%%%%%%%%%%%%%%%%%%
%%%%%%%%%%%%%%%%%%%%%%%%%%%%%%%%%%%%%%%
%%%%%%%%%%%%%%%%%%%%%%%%%%%%%%%%%%%%%%%
\section{Introduction}
Expanders have been extensively studied in the past few decades thanks to their intimate connections with various different fields. It has been used in information theory for the construction of efficient error correcting codes (see e.g. LDPC codes introduced by Gallager~\cite{Gal63} and~\cite{MRRSW20}); and in computer science for randomness extraction~\cite{Gol07} and clustering~\cite{KVV04}, just to name a few. We refer the readers to the survey of Hoory, Linial, and Wigderson~\cite{HLW06} for more comprehensive literature on expanders. The importance of expanders is also reflected by the fact that they have many equivalent definitions. Expanders are the graphs in which every subset of vertices has large boundary. More precisely, given a graph $G$, its \emph{Cheeger constant} is $\min\{\frac{e(A,V(G)\setminus A)}{\mathrm{vol}(A)}: A\subseteq V(G), |A| \leq \frac{1}{2}|V(G)| \}$, where the \emph{volume}, $\mathrm{vol}(A)$, of a set $A$ is the sum of degrees of vertices in $A$. Expanders are graphs with positive Cheeger constant. They can also be defined as the graphs on which the random walk is fast-mixing~\cite{Lov93}, or the ones whose  Laplacian matrices have a spectral gap~\cite{Chee70}.

In reality, having positive Cheeger constant is a very demanding condition as a graph could, at the macroscale, resemble an expander, but have Cheeger constant \emph{zero} due to tiny local noises. In many scenarios (see e.g.~the study of combinatorial maps~\cite{Lou21} and sparse graphs embeddings~\cite{EKK21,crux}), a graph containing a linear-size expander subgraph is as useful as a genuine expander. It is, therefore, of great interest to study the following problem. This natural problem was also proposed by Krivelevich (\cite{K19}, Section~6), see e.g.~\cite{K18, LS20} for work in this direction. 

\begin{problem}\label{main-prob}
	What natural conditions for graphs guarantee a large expander subgraph?
\end{problem}

Yet, another motivation for Problem~\ref{main-prob} comes from the study of sublinear expanders, first introduced by Koml\'os and Szemer\'edi in the 90s~\cite{K-Sz-1}. The theory of sublinear expanders has recently seen a series of advancements, and been involved in the resolutions of several long-standing conjectures~\cite{FKKL,FL21,crux,HHKL21,H-K-L,IKL,KLShS17,LM1,LM2,LWY}. A key lemma of Koml\'os and Szemer\'edi~\cite{K-Sz-1}  asserts that every graph contains a subgraph that is almost as dense and has certain weak expansion properties. There is, unfortunately, no control on how large this expander subgraph must be, as the original graph could be simply a disjoint union of small expanders. Therefore, we could potentially pass from the host graph to a tiny  expander subgraph, which is too small to be useful. This drawback causes lots of difficulties when using sublinear expanders.

In this paper, we study Problem~\ref{main-prob} and show that if the random walks from one of a positive fraction of vertices is well-mixing, then it has an almost spanning induced expander subgraph.

\subsection{Well-mixing vertices, expansion, and separators}\label{sec:mix-exp-sep}
Throughout, all graphs are simple graphs. Let $\G$ be a $D$-regular graph with vertex set $[n] = \{1,2,\ldots,n\}$. Let $\cW = \cW(\G)$ denote the nearest neighbor random walk on $\G$ where in each step, the walk moves to one of the uniformly randomly chosen neighbors (independent of history) of the current vertex. We denote by $Q_v^t$ the distribution of the random walk starting at the vertex $v$ after time step $t$. It is easy to see that if $\G$ is a connected non-bipartite graph, then $Q_i^t(j) \rightarrow \frac{1}{n}$ as $t \rightarrow \infty$, for all $i,j \in [n]$. We use the following notion of distance known as \textit{total variation distance}. For two probability distributions $P_1$ and $P_2$ on $[n]$, define: $$\|P_1 - P_2\| = \max_{A \subseteq [n]} |P_1(A) - P_2(A)| = \frac{1}{2}\sum_{j \in [n]} |P_1(j) - P_2(j)|.$$ % and edge set $E$

Let $U$ denote the uniform distribution on $[n]$, i.e., $U(j) = \frac{1}{n}$ for all $j \in [n]$. We say the random walk $\cW$ mixes after time step $t$, if we have $\|Q_j^t - U\| < \frac{1}{4}$ for all $j \in [n]$ (the choice of $\frac{1}{4}$ is arbitrary and it can be replaced with any small constant). We define the mixing time $\mix(\G)$ to be the smallest such $t$, which will be usually denoted by $\tau$.

It is well-known that small mixing time and edge expansion are qualitatively equivalent. More precisely, the \emph{conductance} of $\G$, denoted by $\phi(\G)$, is:
$$\phi(\G) = \min_{\varnothing\neq A\subsetneq [n]} \frac{n\cdot e(A,[n]\setminus A)}{D\cdot |A| (n-|A|)}.$$
Then large conductance is equivalent to small mixing time~\cite{Lov93}:
\begin{equation}\label{eq:conductance}
	\frac{1}{\phi(\G)}< \mix(\G) < \frac{16\log_2n}{\phi(\G)^2}.
\end{equation}

Similar to having positive Cheeger constant, having small mixing time is also a strict condition due to microscale noise. Indeed, consider the graph $G$ which is a disjoint union of a $D$-regular expander $\G$ and a clique $K_{D+1}$ with $\G$ much larger than the clique. Note that the random walks from any vertex in $\G$ (which are the majority vertices) mixes well rapidly as the uniform distribution on $\G$ is close to the one on $G$ due to the clique being much smaller than $G$. However, $G$ does not have finite mixing time because of the isolated clique $K_{D+1}$. That brings us to the natural guess that if random walks from some vertices are well-mixing, then at least on macroscopic level, the graph should behave like an expander.

\medskip\noindent\textbf{Expander.} Our first result confirms the above suspicion, showing that if the random walks from a positive fraction of starting vertices have small mixing time, then the graph is an expander after possibly deleting a small number of vertices. We informally call a vertex \textit{well-mixing} if the random walk starting from the vertex has small mixing time.

\begin{theorem} \label{thm1}
	Let $\e$ and $\de$ be real numbers satisfying $0 < \e < \frac{2}{5}$ and $0 < \de \le \frac{1}{30}$. Suppose that $\G$ is a $D$-regular $n$-vertex graph on vertex set $[n]$ such that $\|Q^{\tau}_v-U\| < \de$ for at least $\e n$ vertices $v \in [n]$. Then, there exists a set $V'$ of at most $5\de n$ vertices such that the induced graph $\G'$ on $V(\G) \setminus V'$ satisfies the following where $n'=|\G'|$. Every vertex set $X \subseteq V(\G')$ of size at most $\frac{1}{2}n'$ satisfies $e(X, V(\G') \setminus X) \ge \frac{\e D}{16\tau} |X|$. %contains a cycle of length $\ell> \frac{\e n}{250\tau}$.
\end{theorem}

It is worth pointing out that the bound $\frac{\e D}{16\tau}|X|$ we get above relating the mixing time and the edge-expansion in a graph is optimal up to a multiplicative $\log{\tau}$-term. To see this, consider a $D$-dimensional hypercube $Q^D$ on $n=2^D$ vertices, which has mixing time $\tau = O(D\log D) = O(\log n \log\log n )$ from every vertex. Take $X\subseteq V(Q^D)$ to be the vertex subset of a $(D-1)$-dimensional subcube. Then it is not hard to see that $X$ has edge-expansion $e(X,V(Q^D)\setminus X)\leq \frac{D \log{\tau}}{\tau }|X|$.
Also, in the conclusion above, deleting $5\delta n$ vertices cannot be completely omitted, as most of vertices in the disjoint union of a $(1-\delta)n$-vertex expander and a $\delta n$-vertex expander are well-mixing, but the smaller component does not have edge-expansion.

\medskip\noindent\textbf{Separator.} The property of containing well-mixing vertices has also close relation with another central notion, separators, in structural graph theory. Separators measure the connectivity of large vertex sets in a graph. Formally, given an $n$-vertex graph $G$, a vertex set $S\subseteq V(G)$ is a \emph{separator} if the removal of $S$ partitions the graph into disjoint subgraphs each with at most $2n/3$ vertices.

Since the fundamental work of Lipton and Tarjan~\cite{LT79} that every $n$-vertex planar graph has a separator of size $O(\sqrt{n})$, separators have been playing a key role in devising efficient divide and conquer algorithms and dynamic programming, see e.g.~\cite{FLT92}. It is known that having no small separator is qualitatively equivalent to vertex expansion, and it implies that there is a large expander subgraph (\cite{K19}, Section~5). 

As an immediate corollary of Theorem~\ref{thm1}, we see that having linearly many well-mixing vertices is a stronger property, implying in particular that there is no small separator.
\begin{cor}\label{cor:sep}
		Let $\e$ and $\de$ be real numbers satisfying $0 < \e < \frac{2}{5}$ and $0 < \de \le \frac{1}{30}$. Suppose that $\G$ is a $D$-regular $n$-vertex graph on vertex set $[n]$ such that $\|Q^{\tau}_v-U\| < \de$ for at least $\e n$ vertices $v \in [n]$. Then, $\G$ has no separator of size smaller than $\frac{\e}{48\tau} n$.
\end{cor}

The converse above is not true. That is, there are graphs with no small separators and have no well-mixing vertices. Indeed, take two disjoint $n$-vertex $D$-regular expanders and add a perfect matching between them. If $D$ is large, say $D=n^{\Omega(1)}$, then no vertex in the resulting graph has polylogarithmic mixing time, and clearly there is no small separator. 

Corollary~\ref{cor:sep}, together with the celebrated result of Plotkin, Rao and Smith~\cite{PRS94}, also implies that regular graphs with linearly many well-mixing vertices contain large clique minors. We refer the readers to~\cite{KN21} for recent results on edge expansion and clique minors.

\begin{cor}\label{cor:minor}
	Let $\e$ and $\de$ be real numbers satisfying $0 < \e < \frac{2}{5}$ and $0 < \de \le \frac{1}{30}$. Suppose that $\G$ is a $D$-regular $n$-vertex graph on vertex set $[n]$ such that $\|Q^{\tau}_v-U\| < \de$ for at least $\e n$ vertices $v \in [n]$. Then, $\G$ has a shallow clique minor of order $\Omega\Big(\frac{\e}{\tau}\sqrt{\frac{n}{\log n}}\Big)$.
\end{cor}

See Figure \ref{figure} for a diagram that summarizes the connections between well-mixing vertices, expansion, and separators. Note that the qualitative equivalence of edge expansion and having linearly many well-mixing vertices is a consequence of Theorem~\ref{thm1} and~\eqref{eq:conductance}.

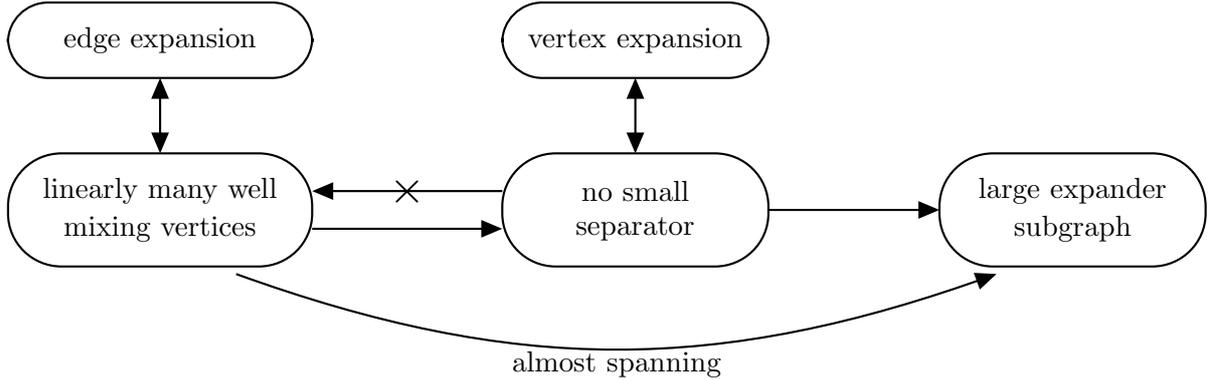
\begin{figure}[h]
\begin{tikzpicture}
\node at (0,0) {linearly many well};
\node at (0,-0.5) {mixing vertices};
\draw [rounded corners=20pt, thick] (-2,-1) rectangle (2,0.5);

\node at (0,2) {edge expansion};
\draw [rounded corners=15pt, thick] (-2,1.5) rectangle (2,2.5);

\node at (6.25,0) {no small};
\node at (6.25,-0.5) {separator};
\draw [rounded corners=20pt, thick] (4.5,-1) rectangle (8,0.5);

\node at (6.25,2) {vertex expansion};
\draw [rounded corners=15pt, thick] (4.5,1.5) rectangle (8,2.5);

\node at (12,0) {large expander};
\node at (12,-0.5) {subgraph};
\draw [rounded corners=20pt, thick] (10.25,-1) rectangle (13.75,0.5);

\draw [>=triangle 45, <->, thick] (0,0.5) -- (0,1.5);
\draw [>=triangle 45, <->, thick] (6.25,0.5) -- (6.25,1.5);
\draw [>=triangle 45, ->, thick] (2,-0.5) -- (4.5,-0.5);
\node at (3.25,0) {$\bigtimes$};
\draw [>=triangle 45, <-, thick] (2,0) -- (4.5,0);
\draw [>=triangle 45, ->, thick] (8,-0.25) -- (10.25,-0.25);

\node at (6,-2.3) {almost spanning};

\draw [>=triangle 45, ->, thick] (1,-1.1) to [out=-20,in=-160] (11,-1.1);

\end{tikzpicture}
\caption{Relations between well-mixing vertices, expansion, and separators} \label{figure}
\end{figure}

%%%%%%%%%%%%%%%%%%%%%%%%%%%%%%%%%%%%%%%
%%%%%%%%%%%%%%%%%%%%%%%%%%%%%%%%%%%%%%%
%%%%%%%%%%%%%%%%%%%%%%%%%%%%%%%%%%%%%%%
\subsection{Long cycles in graphs with well-mixing vertices}
It is a classical NP-complete problem to decide if a graph is Hamiltonian (see, e.g., \cite{D00,GJ79,L79}). To gain some light on this matter, there has been a lot of study to bound the length of the longest path/cycle of a graph. There is neither any polynomial time algorithm known with a reasonable approximation ratio nor any hardness result supporting the situation (see, e.g., \cite{FMS00,KMR97}). Thus, it is interesting to look for natural conditions in graphs which enforce existence of long paths/cycles. 

Pak \cite{P02} proved that a regular graph with small mixing time must contain a large path.

\begin{theorem} [\cite{P02}] \label{P1}
Let $\G$ be a $D$-regular graph with $n$ vertices, $\mix(\G)=\tau$ and $D>8\tau^2$. Then $\G$ contains a path of length $\ell>\frac{n}{16\tau}$.
\end{theorem}

The strategy in \cite{P02} was to probabilistically construct self-avoiding walks and concatenate them to a long path. This method also provides an efficient randomized algorithm for this problem. They further provided a generalization of Theorem \ref{P1} by weakening the mixing time condition to the situation where the random walks from at least $\frac{3}{4}$-fraction of vertices have small mixing time.

\begin{theorem} [\cite{P02}] \label{P2}
Let $\G$ be a $D$-regular graph with $D>8\tau^2$. Suppose $\|Q^{\tau}_v-U\|<\e$ for at least $(1-\n)n$ points $v \in [n]$, where $0 < \e + \n < \frac{1}{4}$. Then $\G$ contains a path of length $\ell> \frac{n}{16\tau}$.
\end{theorem}

As Pak remarked, Theorem~\ref{P1} can be derived using graph theoretic methods via some well-known equivalence between small mixing time and good expansion of graphs. However, quoting from \cite{P02}, ``to the best of our knowledge, Theorem~\ref{P2} cannot be reduced to a similar result on graph expansion''. 

As a corollary of Theorem~\ref{thm1} and a known result on long cycles in graphs with expansion property, we obtain the following algorithmic result, addressing the above comment of Pak.

\begin{theorem} \label{algorithm}
	Let $\e$ be a real number in $(0,\frac{2}{5})$ and $\G$ be a $D$-regular graph with $n$ vertices. Suppose $\|Q^{\tau}_v-U\| < \frac{1}{30}$ for at least $\e n$ points $v \in [n]$. Then, there is a deterministic polynomial time algorithm to find a cycle of length $\ell> \frac{\e n}{40\tau}$ in $\G$.
\end{theorem}

Our result improves Theorem~\ref{P2} in three aspects. First, we require only a positive fraction of vertices to be well-mixing. Second, avoiding probabilistic arguments allows us to have a deterministic algorithm rather than a random one. More importantly, we lift the lower bound condition on the degree $D$. Theorem~\ref{P2} does not apply to constant degree sparse graphs due to the condition on $D$ as even for expanders the mixing time is at least $\tau=\Omega(\log n)$.

\subsection{Positive fraction of well-mixing vertices to almost all}
Our second main result proves that if a positive fraction of vertices are well-mixing, then so are most of the vertices.

\begin{theorem} \label{thm2}
For every $0<\e<1/4$ and a positive integer $M$, there exists $\de_0 = \de_0(\e,M)$ such that the following holds for all $0 < \de < \de_0$. If $\e n$ vertices $v$ of a $D$-regular graph $\G$ on $[n]$ satisfy that $\|Q^{\tau}_v-U\| < \de$, then at least $(1-(\de/\e)^{1/4^M})n$ vertices $v$ of $\G$ satisfy that $\|Q^{(M+1) \tau}_v-U\| < 2e^{-\e M/2} + 6\de$.
\end{theorem}
Note that, by taking $M$ large enough compared to $\e$ and taking $\delta$ small enough compared to $\e$ and $1/M$, both terms $(\de/\e)^{1/4^M}$ and $2e^{-\e M/2} + 6\delta$ can be made arbitrarily small. Note also that, using Theorem~\ref{thm1} and~\eqref{eq:conductance}, one can derive that a positive fraction of vertices having mixing time $\tau$ implies almost all vertices having mixing time $O(\tau^2\log n)$. The value of Theorem~\ref{thm2} is that $M$ is independent of $\tau$, and thus we have in fact the optimal $O(\tau)$ mixing time for almost all vertices.

To prove Theorem~\ref{thm2}, the idea is to show that a short random walk from a typical vertex is very likely to visit one of those well-mixing vertices; we can then conclude from the Markov property that a random walk from most of the vertices mixes rapidly. To this end, we need to  control the set, say $B$, of (bad) vertices from which a short random walk, with decent probability, do not land in the set of well-mixing vertices. It turns out that controlling the bad set $B$ is a delicate matter, and we ended up having to define iteratively layers of bad sets (as one has to also be careful with the set of vertices from which the random walk has good chance to land in $B$).

We remark that the number of well-mixing vertices $\e n$ is best possible up to a factor of $O(\log n)$. To see this, consider the following construction. Suppose $G_1, G_2$ are two regular expanders with degree $D/2$. Let $m = m(n) = \omega(\log n)$. Choose a vertex set $U_i \subseteq V(G_i)$ with size $|U_1| = |U_2| = n/m$. Merge $U_1$ and $U_2$ arbitrarily to obtain a graph $G$ with $2n - n/m$ vertices. Note that the maximum degree of $G$ is at most $D$. Add arbitrary edges to $G_1$ and $G_2$ to make $G$ a $D$-regular graph (adding edges does not affect the edge-expansion properties of an expander). Then, it is not hard to see that a random walk starting from any vertex from the merged vertex set mixes well within $O(\log n)$ steps. However, it is also not hard to show that a random walk starting from most vertices outside of the merged part has mixing time at least $\Omega(m)$. Indeed, observe that for a typical vertex $v \in V(G_i) \setminus U_i$, the random walk in $G_i$ starting from $v$ does not reach $U_i$ within $o(m)$ steps with high probability, hence it does not reach any vertices of $V(G_{3-i})$.

\medskip

\noindent\textbf{Organization.} The rest of this paper is organized in the following manner. The next section contains some notations and simple facts which will be useful throughout this paper. We prove Theorems \ref{thm1} and \ref{algorithm} in Section~\ref{sec:cycle} and prove Theorem \ref{thm2} in Section~\ref{sec:LvsS}.

\section{Preliminaries}
We use the following fact, which is implied by some well-known results (see, e.g., \cite{MS59}, \cite{DG03} or \cite{ES82}). 
Let $W_k$ denote the number of $k$-walks in $G$ with starting vertex specified. In other words, a walk $v_1\dots v_{k+1}$ and $v_{k+1}\dots v_1$ are considered to be different unless $v_{i}=v_{k+2-i}$ for all $i\in [k+1]$. %The following well-known fact about the walks is useful for us.

%For a vertex $v\in V(G)$ in a graph $G$, let $W_k(v)$ denote the number of $k$-walks starting from $v$ and write  $$W_k = \sum_{v \in V(G)} W_k(v).$$

%\begin{lemma}
%For any graph $G$ on $n$ vertices, we have that $W_{2k} \ge \frac{1}{n}W_k^2$.
%\end{lemma}

%\begin{proof}
%Any $2k$-walk with mid-point $v$ can be identified with two $k$-walks starting from $v$. Thus, by Cauchy-Schwarz, we obtain
%$$W_{2k} = \sum_{v \in V(G)} W_k(v)^2 \ge \frac{1}{n} \Big(\sum_{v \in V(G)} W_k(v)\Big)^2 = \frac{1}{n}W_k^2.$$
%\end{proof}

\begin{cor}\label{path}
For any graph on $n$ vertices with average degree $d$, we have that $W_k \ge nd^k$.
\end{cor}

The following is also a well-known fact that taking a random step does not make the distribution further from the stationary distribution. See \cite[Exercise 4.2]{LP17} for reference.
%We use the following exercise in markov chain mixing time. (Exercise 10 of \url{http://www.math.chalmers.se/~steif/exercise2.pdf}, find a better citation)
%\cite[Exercise 4.2]{LP17} (Markov chains and mixing times by D. A. Levin, Y. Peres, E. L. Wilmer)

\begin{lemma} \label{decreasing}
Let $P$ be the transition matrix of a finite Markov chain $(X_t)_{t \in N_0}$ with stationary distribution $\pi$ and starting distribution $\mu$. The total variation distance of the distribution of $X_t$, i.e. $\mu P^t$, to $\pi$ is non-increasing with $t$, i.e. $\|\mu P^t - \pi\| \ge \|\mu P^{t+1} - \pi\|$ for all $t \in N_0$. In particular, for each vertex $v \in V(G)$, the sequence $\|Q^t_v - \pi\|$ is non-increasing in $t$.
\end{lemma}

We use the following result of Krivelevich \cite{K19}.

\begin{theorem} \label{K}
Let $k > 0$ and $\ell \ge 2$ be integers. Let $G$ be a graph on more than $k$ vertices, satisfying: 
\begin{equation}
|N_G(W)| \ge \ell \;\;\;\; \text{for every} \; W \subset V \;\text{with}\; \frac{k}{2} \le |W| \le k.
\end{equation}
Then $G$ contains a cycle of length at least $\ell+1$. Moreover, there is a polynomial time algorithm to find a cycle of length at least $\ell+1$ in $G$. %\footnote{see Thm7.4 in Krivelevich's survey in dropbox}
\end{theorem}

We comment here that the `moreover' part of this Theorem is not explicitly written in \cite{K19}. However, it simply follows from the proof of the original theorem in \cite{K19}. To prove Theorem~\ref{K}, Krivelevich uses the Depth First Search algorithm and some elementary graph operations which can be performed in polynomial time.

For simplicity, we abuse the notation to use $Q_v^k$ also for the random variable recording the endpoint of a random walk of length $k$ starting from a vertex $v$. We often use the following simple fact about random walks which follows from the reversibility property of the underlying markov chain. 

\begin{prop} \label{walknumber}
Let $G$ be a $D$-regular graph. Then for every pair of subsets $A, B \subseteq V(G)$, the following holds for all positive integers $k$. $$\sum_{v \in A} \bP[Q_v^k \in B] = \sum_{v \in B} \bP[Q_v^k \in A].$$
\end{prop}

\begin{proof}
The number of paths of length $k$ starting from $A$ and ending at $B$ can be counted in two ways as follows. $$\sum_{v \in A} D^k \cdot\bP[Q^k_v \in B]=\sum_{v \in B} D^k\cdot \bP[Q^k_v \in A].$$ Thus, the proposition follows from the above equality.
\end{proof}

We find the following notion convenient to work with. 
\begin{defn}
	A vertex $v$ is \emph{$(\tau,\e,\de)$-mixing} if for all but at most $\de n$ vertices $u \in V(G)$, $$\frac{1 - \e}{n} \le \bP[Q_v^{\tau} = u] \le \frac{1 + \e}{n}.$$
\end{defn}

The following fact will be handy in proving Theorem \ref{thm2}. %We first relate this notion with the standard notion using total variation distance to show their equivalence in some sense. \todo{when defining new notion, use emph to emphasize}

\begin{prop} \label{newnotion}
For every vertex $v$ in a graph $\G$, the following holds. If $\|Q^{\tau}_v-U\| \le \de$, then $v$ is $(\tau,\e,\frac{2\de}{\e})$-mixing.
%If $v$ is $(t,\e,\de)$-mixing, then $\|Q^t_v-U\|<\e + \de$.
\end{prop}

\begin{proof}
For the sake of contradiction, assume that the vertex $v$ is not $(\tau,\e,\frac{2\de}{\e})$-mixing. Then, there are more than $\frac{2\de n}{\e}$ vertices $u \in V(G)$ such that $\left|\bP[Q_v^{\tau} = u] - \frac{1}{n}\right| > \frac{\e}{n}$. Thus, we have that $\|Q^{\tau}_v-U\| > \frac{1}{2} \cdot \frac{2\de n}{\e} \cdot \frac{\e}{n} =  \de$, a contradiction.
\end{proof}

\section{Small well-mixing set and long cycles}\label{sec:cycle}
%\begin{proof}[Proof of Theorem~\ref{algorithm}]
%\end{proof}
To prove Theorem~\ref{thm1}, we start by showing that it is enough to prove the following with a connectedness assumption in the underlying graph $\G$ in Theorem~\ref{thm1}.

\begin{theorem} \label{connected}
Let $\e$ and $\de$ be real numbers satisfying $0 < \e < \frac{2}{5}$ and $0 < \de \le \frac{1}{20}$. Suppose that $\G$ is a connected $D$-regular $n$-vertex graph such that $\|Q^{\tau}_v-U\| < 2\de$ for at least $\e n$ vertices $v \in [n]$. Then, every vertex set $X \subseteq V(\G)$ of size in $[4\de n, \frac{1}{2}n]$ satisfies $e(X, V(\G) \setminus X) \ge \frac{\e D}{8\tau} |X|$. %$\G$ contains a cycle of length $\ell> \frac{\e n}{200\tau}$.
\end{theorem}

It is immediate to see that Theorem~\ref{algorithm} is implied by Theorems~\ref{connected} and \ref{K}. Indeed, the existence of a vertex $v$ with $\| Q_v^{\tau}-U\|<\frac{1}{20}$ implies that the graph $\G$ given in Theorem~\ref{algorithm} contains a component of size at least $\frac{19}{20}n$, which can be found in polynomial time by e.g. performing a Depth First Search. We can then apply Theorems~\ref{connected} and \ref{K} on this large component to find the desired long cycle.

Let us now show how to prove Theorem \ref{thm1} by using Theorem \ref{connected}. For the convenience of our analysis, to distinguish the underlying graphs, we write $Q^t_v(G)$ for the distribution of the random walk $Q^t_v$ in graph $G$. We write $U(G)$ to denote the uniform distribution on the vertex set of $G$. 

\begingroup
\def\thetheorem{\ref{thm1}}
\begin{theorem} 
Let $\e$ and $\de$ be real numbers satisfying $0 < \e < \frac{2}{5}$ and $0 < \de \le \frac{1}{30}$. Suppose that $\G$ is a $D$-regular $n$-vertex graph such that $\|Q^{\tau}_v-U\| < \de$ for at least $\e n$ vertices $v \in [n]$. Then, there exists a set $V'$ of at most $5\de n$ vertices such that the induced graph $\G'$ on $V(\G) \setminus V'$ satisfies the following where $n'=|\G'|$. Every vertex set $X \subseteq V(\G')$ of size at most $\frac{1}{2}n'$ satisfies $e(X, V(\G') \setminus X) \ge \frac{\e  D}{16\tau} |X|$. 
\end{theorem}
\addtocounter{theorem}{-1}
\endgroup
\begin{proof}
Suppose that $\G$ is an arbitrary $D$-regular graph on $[n]$ with the property that $\|Q^{\tau}_v (\G) - U(\G)\|<\de$ for at least $\e n$ vertices $v \in [n]$. Observe that each of these $\e n$ vertices must lie in a connected component $\G^*$ with $n^*\geq (1-\de)n$ vertices. For each of these vertices $v$, $Q^{\tau}_v (\G)=Q^{\tau}_v (\G^*)$. Thus,
\begin{equation}\label{eq:heart}
	\|Q^{\tau}_v (\G^*) - U(\G^*)\| < \|Q^{\tau}_v (\G) - U(\G)\| + \|U(\G) - U(\G^*)\| < \de + \de = 2\de.
\end{equation}
Therefore, one can apply Theorem~\ref{connected} on $\G^*$ to conclude the following.
\begin{equation}\label{eq: 3}
	\begin{minipage}{0.9\textwidth}
	Any vertex set $X\subseteq V(\G^*)$ with size in $[4\delta n^* , \frac{1}{2} n^*]$ satisfies $e(X,V(\G^*)\setminus X)\geq \frac{\e  D}{8\tau}|X|$.
\end{minipage}
\end{equation}
Let $Y\subseteq V(\G^*)$ be a maximal vertex set satisfying $|Y|\le \frac{1}{2}n^*$ and $e(Y,V(\G^*)\setminus Y) < \frac{\e  D}{16\tau}|Y|$. If such a set does not exist, let $Y=\varnothing $. Then we have $|Y|< 4\delta n^*$ by \eqref{eq: 3}. Let $\G'$ be the graph obtained from $\G^*$ by deleting vertices in $Y$ and let $n'=|\G'|$.

Suppose that $\G'$ contains a nonempty vertex set $X$ with $e(X,V(\G')\setminus X) < \frac{\e  D}{16\tau}|X|$ and $|X|\le  \frac{1}{2} n'$. Then, by letting $Z=V(\G')\setminus (X\cup Y)$, we have 
\begin{align}
e(X\cup Y,Z) &\leq e(Y, V(\G^*)\setminus Y)+ e(X,V(\G')\setminus X)
< \frac{\e  D}{16\tau}|Y| + \frac{\e  D}{16\tau}|X| =\frac{\e D}{16\tau}|X\cup Y|. %\min\{ |X\cup Y| , |V(\G^*)\setminus(X\cup Y)| \}.
\label{eq:why}
\end{align}
If $|X\cup Y|\leq \frac{1}{2}n^*$, this contradicts the maximality of $Y$, so we have $|X\cup Y| > \frac{1}{2}n^*$. On the other hand, as $|Z| = n^*-|X\cup Y|$, we have 
$$|X\cup Y| \leq \frac{1}{2}(n^*-|Y|)+ |Y| \leq \left(\frac{1}{2}+2\delta\right)n^*\leq \frac{2}{3} n^* \enspace \text{and} \enspace \frac{1}{3}n^*\leq |Z|  < \frac{1}{2}n^*,$$ 
and 
$$ \frac{\e D}{16\tau}|X\cup Y|  \leq \frac{\e D}{8\tau }|Z|.$$ 
As $4\delta n^*< \frac{1}{3}n^* \leq |Z|< \frac{1}{2} n^*$, this together with \eqref{eq:why} contradicts \eqref{eq: 3}. Hence, any vertex set $X\subseteq \G'$ of size at most $\frac{1}{2}n'$ satisfies $e(X, V(\G') \setminus X) \ge \frac{\e D}{16\tau} |X|$ in $\G'$ as desired.
\end{proof}

We finally finish by proving Theorem~\ref{connected}.
%\begin{claim} \label{cl1}
%For all $D$-regular connected graph $\G$ on $n$ vertices with at least $\e n$ vertices $v$ satisfying $\|Q_v^\tau - U\| \le \frac{1}{10}$, every vertex set $X$ of size in $[\frac{n}{5}, \frac{n}{2}]$ satisfies $e(X, V \setminus X) \ge \frac{\e D}{40\tau} |X|$. 
%Let $\e$ be a real number in $(0,\frac{2}{5})$ and $\G$ be a $D$-regular graph. Suppose $\|Q^{\tau}_v-U\| < \frac{1}{20}$ for at least $\e n$ points $v \in [n]$. Then $\G$ contains a cycle of length $\ell> \frac{\e n}{250\tau}$.
%\end{claim}Theorem~\ref{connected} follows from Claim \ref{cl1} by using Theorem~\ref{K}. 

\begin{proof}[Proof of Theorem~\ref{connected}]
We start with a sketch of the proof of Theorem~\ref{connected}. Call the vertices $v$ good if they satisfy $\|Q_v^\tau - U\| < 2\de$. Suppose there is a set $X$ of size in $[4\de n, \frac{1}{2}n]$ with a small edge-expansion, in particular $e(X, V \setminus X) <  \frac{\e  D}{8\tau} |X|$. We show that there would be a good vertex $v$ such that the random walk starting at $v$ after time $\tau$ lands at the part that contains $v$ with large probability. This would contradict the fact that $v$ is a good vertex.

%Let $W_k(v)$ denote . Define $W_k = \sum_{v \in X} W_k(v)$.
Consider the subgraph $G$ of $\G$ induced by $X$. By the choice of $X$, the average degree in $G$ is at least $\left(1-\frac{\e}{8\tau}\right)D$.
%Choose the least integer $k \ge \tau$ that is a power of $2$.We clearly have $k < 2\tau$. 
Applying Corollary~\ref{path} to $G$, we get that the number of $\tau$-walks starting from $X$ and which always stay in $X$, denoted by $W_\tau$, satisfies that $$W_\tau \ge \left(1-\frac{\e }{8\tau}\right)^\tau D^\tau |X|.$$

%$W_k \ge \left(1 - \frac{1}{\log^5 n}\right)^k d^k |X| - \frac{k}{\log^5 n} d^k |X|$.

Now let $p$ denote the probability that the random walk of length $\tau$ from a uniformly chosen random vertex $v$ in $X$ ends up in $X$. As there are $D^\tau |X|$ walk of length $\tau$ from a vertex in $X$, We have the following bound. 

\begin{align} \label{bound:p}
p &\ge  \left(1-\frac{\e}{8\tau}\right)^\tau > 1 - \frac{\e}{4}.
\end{align}

%By the connectedness of $\G$, we know that the stationary distribution of the random walk on $\G$ is the uniform distribution $U$. Thus, by Lemma~\ref{decreasing}, we have that $\|Q_v^k - U\| \le \|Q_v^\tau - U\| < 2\de$ for each of the good vertices $v$, and so 
By the assumption that  $\|Q_v^\tau - U\| < 2\de$ for each of the good vertices $v$, we have 
\[\Pr[Q_v^\tau \in A] \ge \frac{|A|}{n} - 2\de\;\;\text{for any}\;\;A \subseteq V(G).\]

Let $I$ denote the set of good vertices of $\G$. It is clear that either $X$ or $V \setminus X$ contains half of the vertices of $I$. First assume that $|I \cap X| \ge \frac{\e n}{2}$. Then we obtain the following: %(else, work instead with $V \setminus X$). 
\begin{align*}
1 - p &= \frac{1}{|X|} \sum_{v \in X} \bP[Q_v^\tau \in V \setminus X] \ge \frac{1}{|X|} \sum_{v \in I \cap X} \bP[Q_v^\tau \in V \setminus X] \\
&\ge \frac{1}{|X|} |I \cap X| \left(\frac{|V \setminus X|}{n} - 2\de\right) \ge \frac{1}{n/2} \cdot \frac{\e n}{2} \cdot \left(\frac{1}{2} - \frac{2}{20}\right) > \frac{\e}{4},
\end{align*}
which contradicts \eqref{bound:p}. Similar contradiction is achieved if $|I \cap (V \setminus X)| \ge \frac{\e n}{2}$ as follows. By Proposition \ref{walknumber}, we have
\begin{align*}
1 - p &= \frac{1}{|X|} \sum_{v \in V \setminus X} \bP[Q_v^\tau \in X] \ge \frac{1}{|X|} \sum_{v \in I \cap (V \setminus X)} \bP[Q_v^\tau \in X] \\
&\ge \frac{1}{|X|} |I \cap (V \setminus X)| \left(\frac{|X|}{n} - 2\de\right) > \frac{\e n/2}{|X|} \cdot \frac{|X|}{2n} = \frac{\e}{4}.
\end{align*}
This completes the proof of Theorem \ref{connected}.% and \ref{thm1}.
\end{proof}

\section{Large well-mixing set vs small well-mixing set}\label{sec:LvsS}
In this section, we prove Theorem \ref{thm2}.

%\begin{theorem} \label{convenient}
%For every $\e > 0$, there exists $M_0 = M_0(\e)$ such that the following holds for all $M > M_0$. There exists $\de_0 = \de_0(\e, M)$ such that whenever $0 < \de < \de_0$ and $\e n$ vertices of a $D$-regular graph $\G$ are $(\tau,\e,\de)$-mixing, then at least $(1-2\de^{1/4^M})n$ vertices of $\G$ are $(2M \tau, 2\e, \de)$-mixing. 
%\end{theorem}

\begingroup
\def\thetheorem{\ref{thm2}}
\begin{theorem}
For every $0<\e<1/4$ and a positive integer $M$, there exists $\de_0 = \de_0(\e,M)$ such that the following holds for all $0 < \de < \de_0$. If $\e n$ vertices $v$ of a $D$-regular graph $\G$ on $[n]$ satisfy that $\|Q^{\tau}_v-U\| < \de$, then at least $(1-(\de/\e)^{1/4^M})n$ vertices $v$ of $\G$ satisfy that $\|Q^{(M+1) \tau}_v-U\| < 2e^{-\e M/2} + 6\de$.\end{theorem}
\addtocounter{theorem}{-1}
\endgroup
\begin{proof}
Let $\delta< \delta_0 = \frac{1}{10^8}\e \cdot 3^{-M\cdot4^M}$.
Let $A$ be the set of all vertices $v$ of $\G$ satisfying that $\|Q^{\tau}_v-U\| < \de$. By the assumption on $\G$, we have that $|A| \ge \e n$. By Proposition \ref{newnotion}, the vertices of $A$ are $(\tau,\e,\frac{2\de}{\e})$-mixing. 

Collect the vertices $u$ where the random walks from $u$ is unlikely to visit $A$ and let  $B_0$ be the set of such vertices as follows.
$$B_0 = \{v : \bP[Q^\tau_v \in A] < \e/2\}.$$

\begin{claim}\label{claim:B0}
We have that $|B_0| < \frac{6\de n}{\e}$.
\end{claim}

\begin{poc}
%The number of paths of length $\tau$ starting from $A$ and ending at $B_0$ is given by 
By Proposition \ref{walknumber}, we have the following:
\begin{equation} \label{identity}
\sum_{v \in A} \bP[Q^\tau_v \in B_0]=\sum_{v \in B_0} \bP[Q^\tau_v \in A].
\end{equation}
By the definition of $B_0$, we have the following:% for every $v \in B_0$.
\begin{equation} \label{a}
\bP[Q^\tau_v \in A] < \e/2\;\text{for every}\;v\in B_0.
\end{equation}
Suppose for the sake of contradiction that $|B_0| \ge \frac{6\de n}{\e}$. Then, since the vertices of $A$ are $(\tau,\e,\frac{2\de}{\e})$-mixing, for every $v \in A$, we have that
\begin{equation} \label{b0}
\bP[Q^\tau_v \in B_0] = \sum_{u \in B_0} \bP[Q^\tau_v = u] \ge \left(|B_0| - \frac{2\de n}{\e} \right) \cdot \frac{1-\e}{n} \ge \frac{2|B_0|}{3} \cdot \frac{1-\e}{n}.
\end{equation}
%Thus, \eqref{identity}, \eqref{a}, and \eqref{b0} contradict the fact that $|A| \ge \e n$. Hence, $|B_0| < 3\de n$.
Combining \eqref{identity}, \eqref{a}, \eqref{b0} and the assumption that $|A| \geq \epsilon n$, we obtain
\[\frac{\epsilon}{2}|B_0| > \frac{2|B_0|}{3}\cdot\frac{1-\epsilon}{n}|A| \geq \frac{2}{3}|B_0|(1-\epsilon)\epsilon.\]
This is a contradiction to the fact $\epsilon < 1/4$.
Therefore, it must be that $|B_0| < \frac{6\de n}{\e}$.
\end{poc}

We wish to show that a random walk of length $M\tau$ from most of the vertices visit $A$ at least once. However, it would be problematic if such a random walk visit $B_0$ before $A$. In order to show that visiting $A$ before $B_0$ is more likely for most of the starting vertices, we aim to define the (bad) sets $B_1,\dots, B_M$ in such a way that $B_{i}$ is the set of vertices that the random walk from it is likely to visit $\bigcup_{j=0}^{i-1} B_j$ before $A$.
It turns out to be more convenient to work with the random walk in `chunks' of length $\tau$, that is the Markov chain $\{Q^{i\tau}_v\}_{i\in \mathbb{N}}$. We shall define all the preceding sets $B_i$ with respect to this Markov chain.  

To control the size of those sets, we let $\h_i = 2\cdot3^{i+1} (\de/ \e)^{2^{-i}}$ for each $0\leq i\leq M$. Observe that $\h_i \le \frac{1}{9}\h_{i+1}^2$. Next we inductively define $B_i$ and $B^i$ as follows. Let $B^0 = B_0$, and for each $1\le i \le M$, let $$B_i = \{v : \bP[Q^\tau_v \in B^{i-1}] > \h_i\} \;\;\;\;\;\; \text{and} \;\;\;\;\;\; B^i = \bigcup_{j=0}^i B_j.$$

\begin{claim} \label{bm}
We have that $|B^i| <  \h_i n$ for all $i \le M$.
\end{claim}

\begin{poc}
We use induction on $i$. The base case $i=0$ holds by Claim~\ref{claim:B0}.
Assume that $|B^{i-1}| < \h_{i-1} n$. The number of paths of length $\tau$ starting from $B^{i-1}$ is at most $D^\tau |B^{i-1}| \le D^{\tau} \h_{i-1} n$. On the other hand, the number of paths of length $\tau$ from $B^{i-1}$ to $B_i$ is at least the following:
\begin{equation*}
\sum_{v \in B_i} D^{\tau} \bP[Q^\tau_v \in B^{i-1}] \ge \sum_{v \in B_i} D^{\tau} \eta_i \ge |B_i| D^{\tau} \eta_i. 
\end{equation*}
Thus, we have that $$|B_i| \le \frac{\h_{i-1} n}{\h_i} \le \frac{\frac{1}{9} \h_i^2}{\h_i} n \le \frac{1}{9} \h_i n.$$ Hence, we have that $|B^i| \le |B^{i-1}| + |B_i| < \h_i n$, proving our claim.
\end{poc}

As $\delta< \frac{1}{10^8}\e \cdot 3^{-M\cdot4^M}$, Claim \ref{bm} yields us that $|B^M| < \eta_M n < (\de/\e)^{1/4^M} n$. We will show that all the vertices outside of $B^M$ are indeed well-mixing. To this end, fix a vertex $v$ outside of $B^M$. 
We want to show that the Markov chain $\{Q^{i\tau}_v\}_{i\in \mathbb{N}}$ is more likely to visit $A$ before it reaches $B_0$. We will show that, it is unlikely that this chain skips one of $B_M, B_{M-1},\dots, B_1$ before it reaches $B_0$, and this provides many chances to visit $A$ before visiting $B_0$.

Denote by $E_i$ the event that $Q^{i\tau}_v \in A$ and denote by $F_i$ the event that $Q^{i\tau}_v \notin B^{M-i}$. 
By the definition of the sets $B^i$, it follows  that for all $0\le i \le M-1$,
\begin{equation}\label{eq:chicken}
\bP[\overline{F}_{i+1} | \cap_{j \le i} F_j] \le \h_{M-i}.
\end{equation}
Next note that if the event $F_i$ happens, then $Q^{i\tau}_v \notin B_0$. 
%Also, $\bP[Q^{\tau}(Q^{i\tau}_v) \in A | Q^{i\tau}_v \notin B_0] \ge \e/2$. Thus, we have the following.
By the definition of $B_0$, it must be that $\bP[Q^{\tau}_u=Q^{(i+1)\tau}_{v}  \in A | Q^{i\tau}_v=u \text{ for some } u\notin B_0] \ge \e/2$.
Thus, we have the following for all $0\le i\le M-1$:
\begin{equation}\label{eq:burger}
\bP[E_{i+1} | \cap_{j \le i} (\overline{E}_j \cap F_j)] \ge \e/2,
\end{equation}
where the $i=0$ case follows from $v\not\in B^0 \subseteq B^M$.
Using~\eqref{eq:chicken} and~\eqref{eq:burger}, we can then bound the probability of the event that $Q^{i\tau}_v$, $1\le i \le M$, never land in $A$ as follows:
\begin{align*}
\bP[\cap_{i=1}^M \{Q^{i\tau}_v \notin A\}]&=\bP\left[\cap_{i=1}^M \overline{E}_i\right] \\
&\le \bP\left[\cup_{i=1}^M \overline{F}_i\right] + \bP\left[\left(\cap_{i=1}^M \overline{E}_i\right) \bigcap \left(\cap_{i \le M} F_i\right)\right] \\
&= \sum_{i=1}^M \bP\left[\;\overline{F}_i | \cap_{j \le i-1} F_j\right] + \prod_{i=1}^M \bP\left[\;\overline{E}_i \cap F_i | \cap_{j \le i-1} (\overline{E}_j \cap F_j)\right] \\
&\le \sum_{i=1}^M \h_{M-i+1} + \prod_{i=1}^M \bP\left[\;\overline{E}_i | \cap_{j \le i-1} (\overline{E}_j \cap F_j)\right] \\
&\le \sum_{i=1}^M \h_{M-i+1} + (1-\e/2)^M \\
&\le 2e^{-\e M/2}. %< \e^{100}.
\end{align*}
The last inequality follows from our choice of $\delta$ that $(\de/\e)^{1/4^M}\leq 3^{-M}< e^{-\e M/2}$. 
Thus, with probability at least $1 - 2e^{-\e M/2}$, the sequence $\{Q^{i\tau}_v\}_{i=1}^M$ has a vertex from $A$. Denote this event by $\E$. %From now on, we assume that the underlying graph $\G$ is connected (we can easily handle the case when $\G$ is not connected like in the last section). 

For $u \in A$ and $1\le i \le M$, define the event 
$$\E_{i,u}=\Big\{Q^{i\tau}_v = u \text{ and } u \text{ is the first vertex to appear from $A$ in the sequence } \{Q^{i\tau}_v\}_{i=1}^M\Big\}.$$ 
Then, by definition, $\E$ is the disjoint union of $\E_{i,u}$'s. Let 
$$p_{i,u} = \bP[\E_{i,u}] \quad \text{and} \quad p = \sum_{1\le i \le M, u \in A} p_{i,u}=\bP(\E).$$ 
It is clear that $p \ge 1 - 2e^{-\e M/2}$. We now estimate the probability that  $Q^{(M+1)\tau}_v $ lands at a vertex $w$ in $\G$. By the Markov property:
\begin{align*}
\bP\left[Q^{(M+1)\tau}_v = w\right] &= \sum_{1\le i \le M, u \in A} p_{i,u} \cdot \bP\left[Q^{(M+1)\tau}_v = w | Q^{i\tau}_v = u\right] + (1-p)\cdot \bP\left[Q^{(M+1)\tau}_v = w | \E^c\right] \\
&= \sum_{1\le i \le M, u \in A} p_{i,u}\cdot  \bP\left[Q^{(M+1-i)\tau}_u = w\right] + (1-p)\cdot \bP\left[Q^{(M+1)\tau}_v = w | \E^c\right].
\end{align*}
As $\sum_{w\in V(\G)} \bP[Q^{(M+1)\tau}_v = w |  \E^c] =1$, by the triangle inequality, we deduce that
\begin{align}
\|Q^{(M+1)\tau}_v - U\| &= \frac{1}{2} \sum_{w \in V(\G)} \left|\bP\left[Q^{(M+1)\tau}_v = w\right] - \frac{1}{n}\right| \nonumber \\
&\le (1-p) + \frac{1}{2} \sum_{w \in V(\G)} \sum_{1\le i \le M, u \in A} p_{i,u}\cdot \left|\bP\left[Q^{(M+1-i)\tau}_u = w\right] - \frac{1}{n}\right| \nonumber \\
&= 2e^{-\e M/2} + \sum_{1\le i \le M, u \in A} p_{i,u}\cdot  \|Q^{(M+1-i)\tau}_u - U\| \label{finalbound}. 
%&\le \e^{100} + \e \de.
\end{align}
As $A$ is a nonempty set, it is easy to see that the largest connected component $\G^*$ of $\G$ contains at least $(1-\de)n$ vertices with $A\subseteq \G^*$ and $\|U(\G^*)-U\|\leq  2\delta$.
Invoking Lemma \ref{decreasing} and recalling that $u\in A$, we can use the triangle inequality to obtain that for each $t \ge \tau$,
\begin{align*}
\|Q^t_u - U\| &\le\|Q^t_u - U(\G^*) \| + \|U(\G^*) - U\| \leq \|Q^{\tau}_u - U(\G^*)\|+ \|U(\G^*) - U\| \\ &\leq \|Q^{\tau}_u - U\|+ \|U - U(\G^*)\|+ \|U(\G^*) - U\| 
\le 6  \de.	
\end{align*}
Using this in \eqref{finalbound}, we get that $\|Q^{(M+1)\tau}_v - U\| < 2e^{-\e M/2} + 6\de$. %proving Theorem \ref{thm2}. %If $\G$ is not connected, then similar to~\eqref{eq:heart}, $\| Q^t_u - U\| \leq \|Q^t_u-U(\G^*)\|+ \|U(\G^*) - U(\G)\| \leq \e \del$. $\Gamma^*$ be the connected component of , one can deduce that $\G$ contains a connected component $\G'$ of size close to $n$ and run the same argument in $\G'$. 
This finishes the proof of Theorem \ref{thm2}.
\end{proof}
 
 \section*{Acknowledgment}
 We thank an anonymous reviewer for suggesting the connection between well-mixing vertices and separators in Section~\ref{sec:mix-exp-sep} and for helping us improve the exposition of this paper.

\end{document}